\documentclass[a4paper,12pt,reqno]{amsart}
\usepackage{amsmath,amsthm,amssymb,latexsym,epsfig,graphicx,subfigure}
\numberwithin{equation}{section}
\theoremstyle{definition}
\newtheorem{thm}{Theorem}[section]
\theoremstyle{definition}
\newtheorem{lm}[thm]{Lemma}
\theoremstyle{definition}
\newtheorem{cor}[thm]{Corollary}
\theoremstyle{definition}
\newtheorem{pr}[thm]{Proposition}
\theoremstyle{definition}
\newtheorem{df}[thm]{Definition}
\theoremstyle{remark}
\newtheorem{rem}[thm]{Remark}

\newcommand{\Rn}{\mathbb{R}^{n}}
\newcommand{\hn}{\mathbb{R}^{n-1}}
\newcommand{\R}{\mathbb{R}}

\newcommand{\T}{\mathbb{T}}
\newcommand{\N}{\mathbb{N}}
\renewcommand{\H}{\mathbb{H}}
\renewcommand{\hn}{\mathbb{H}^n}

\newcommand{\V}{\mathcal{V}}
\newcommand{\W}{\mathcal{W}}

\renewcommand{\H}{\mathcal{H}}

\newcommand {\grtrsim} {\ {\raise-.5ex\hbox{$\buildrel>\over\sim$}}\ }

\newcommand{\e}{\varepsilon}
\renewcommand{\r}{\rho}

\newcommand{\ra}{\rightarrow}
\newcommand{\khii}{\text{\lower -.4ex\hbox{$\chi$}}}
\DeclareMathOperator{\spt}{spt} \DeclareMathOperator{\dist}{dist}
 
 \DeclareMathOperator{\diam}{diam}
\DeclareMathOperator{\tanm}{Tan} \DeclareMathOperator{\tr}{tr}
\DeclareMathOperator{\ittanm}{itTan}

\renewcommand{\t}{\tau}
\newcommand{\f}{\varphi}
\renewcommand{\a}{\alpha}

\newcommand{\rsei}{\mathcal{R}^{\varepsilon}_{s,i}}
\newcommand{\rsim}{\mathcal{R}^*_{s,i}}
\newcommand{\rsm}{\mathcal{R}^{*}_s}

\begin{document}
\title [Singular integrals on AD-regular subsets of the Heisenberg group] {Singular integrals on Ahlfors-David regular subsets of
the Heisenberg group}
\author{Vasilis Chousionis and Pertti Mattila}

\thanks{The authors are grateful to the Centre de Recerca Matem\'atica, Barcelona, for providing an excellent working enviroment during the thematic semester ``Harmonic Analysis, Geometric Measure Theory and Quasiconformal Mappings'', where parts of this paper were written. Both authors were supported by the Academy of Finland.} 

\begin{abstract}We investigate certain singular integral operators with Riesz-type kernels on
s-dimensional Ahlfors-David regular subsets of Heisenberg groups. We show
that $L^2$-boundedness, and even a little less, implies that $s$ must be an integer and the set
can be approximated at some arbitrary small scales by homogeneous 
subgroups. It follows that the operators cannot be bounded on many
self similar fractal subsets of Heisenberg groups.
\end{abstract}

\maketitle

\section{Introduction}
In this paper we study certain singular integral operators on lower dimensional, in terms of Hausdorff dimension, 
subsets of the Heisenberg group $\hn$. We shall first review analogous results which are known to be true in $\R^n$. 
We shall study Ahlfors-David regular and somewhat more general sets and measures:

\begin{df}\label{AD}
A Borel measure $\mu$ on a metric space $X$ is Ahlfors-David regular, or 
AD-regular, if for some positive numbers $s$ and $C$,
$$r^s/C \leq \mu(B(x,r)) \leq Cr^s\ \text{for all}\ x\in \spt\mu, 0<r<\diam(\spt\mu),$$
where $\spt\mu$ stands for the support of $\mu$.
\end{df}

In $\R^n$ the most well known relevant singular integral operators for such $s$-dimensional AD-regular measures are those defined 
by the vector-valued Riesz kernel $|x|^{-s-1}x, x\in\R^n$. The basic question is the validity of the 
$L^2$-boundedness:

\begin{equation}
\label{L2}
\int \left|\int_{X\setminus B(x,r)}\frac{x-y}{|x-y|^{s+1}}g(y)d\mu y\right|^2d\mu x \leq C\int|g|^2d\mu 
\end{equation}
for all $g\in L^2(\mu)$ and all $r>0$. 

Vihtil\"a showed in \cite{V} that if this $L^2$-boundedness holds for some non-trivial $s$-dimensional AD-regular 
measure in $\R^n$, then $s$ must be an integer. Moreover, it was shown in \cite{MPa} and \cite{M4} that in this case 
$\mu$ can be approximated almost everywhere at some arbitrarily small scales by Hausdorff $s$-dimensional measures 
on $s$-planes. In our main theorem, Theorem \ref{mthm}, we prove natural analogues of these results in $\hn$.

It is an open question in $\R^n$, and will remain as such also in $\hn$, whether above 'some arbitrarily small scales' could be replaced by 'all sufficiently small scales'. This would mean that $\mu$ would be a rectifiable measure. Even more could be expected: for AD-regular sets in $\R^n$ the $L^2$-boundedness could be equivalent to the uniform rectifiability in the sense of David and Semmes, the converse is known to hold, see \cite{DS}. This equivalence is valid for 1-dimensional AD-regular sets in $\R^n$ by \cite{MMV}. The proof is based on a relation between the Riesz kernel, $|x|^{-2}x, x \in \Rn$, and the so-called Menger curvature, see e.g. \cite{Pa} and \cite{Ve}. It is not clear to us if a similar strategy can be followed in our setting as we don't know if an analogous relation holds true in the Heisenberg group.

In the last section of the paper we discuss some self-similar sets to which our results apply. First we modify 
ideas of Strichartz from \cite{St} to construct standard Cantor sets on which the Riesz transforms cannot be 
$L^2$-bounded. These are kind of analogues of the Garnett-Ivanov Cantor sets (see \cite{To}) which many authors have 
used as examples to study and illustrate Cauchy and Riesz transforms and analytic and harmonic capacities. In 
fact, we shall prove the non-boundedness of our Riesz transforms on more general self-similar sets.

\section{Notation and Setting}

For an introduction to Heisenberg groups, see for example \cite{cap} or \cite{blu}. Below we state the basic facts needed in this paper.

The Heisenberg group $\hn$, identified with $\R^{2n+1}$, is a non-abelian group where the group
operation is given by,
$$p\cdot q=(p_1+q_1,..,p_{2n}+q_{2n},p_{2n+1}+q_{2n+1}+A(p,q)),$$
where
$$A(p,q)=-2\sum_{i=1}^n(p_iq_{i+n}-p_{i+n}q_i).$$
We will also denote points $p \in \hn$ by $p=(p',p_{2n+1}), p'\in\R^{2n}, p_{2n+1}\in\R$.
For any $q \in \hn$ and $r >0$, let $\tau_q:\hn \ra \hn$ be the left translation
$$\tau_q(p)=q\cdot p,$$
and define the dilation $\delta_r:\hn \ra \hn$ by
$$\delta_{r}(p)=(rp_1,..,rp_{2n},r^2p_{2n+1}).$$
These dilations are group homomorphisms.

A natural metric $d$ on $\hn$ is defined by
$$d(p,q)=\|p^{-1}\cdot q\|$$
where
$$\|p\|=(\|(p_1,..,p_{2n})\|^4_{\R^{2n}}+p^2_{2n+1})^{\frac{1}{4}}.$$
The metric is left invariant, that is $d(q\cdot p_1,q\cdot p_2)=d(p_1,p_2)$, and the dilations satisfy 
$d(\delta_r(p_1),\delta_r(p_2))=rd(p_1,p_2)$. The closed and open balls with respect to $d$ will 
be denoted by $B(p,r)$ and $U(p,r)$. Moreover, we use the notation $B(r)$ and $U(r)$ when the centre $p$ 
is the origin $0$, which is the neutral element of the group.
The  Euclidean metric on $\hn$ will be denoted by $d_E$.

A subgroup $G$ of $\hn$ is called homogeneous if it is closed and invariant under 
the dilations; $\delta_r(G)=G$ for all $r>0$. Every homogeneous subgroup $G$ is a linear subspace of $\R^{2n+1}$. 
We call $G$ a $d$-subgroup if its linear dimension $\dim G$ is $d$. 

We denote
$$\T=\{p \in \hn :p'=0\}\ \text{and}\ H=\{p \in \hn:p_{2n+1}=0\}.$$
Then $T$ is a homogeneous subgroup but $H$ is not a subgroup. We shall often identify $H$ with $\R^{2n}$.

If $V$ is a $d$-subgroup of $\hn$, define the cone
$X(p,V,\delta)$ for $p \in \hn$ and $\delta \in (0,1)$ as
\begin{equation}
\label{cone} X(p,V,\delta):=\{q \in \hn : \dist(p^{-1}\cdot q,
V)<\delta d(q,p)\}.
\end{equation}
It follows that $X(p,V,\delta)=p\cdot X(0,V,\delta)$.

We shall denote by $G(m,k)$ the Grassmannian of $k$-dimensional subspaces of $\R^m$ and  
$G(m):=\cup_{k=0}^m G(m,k)$. Then $G(m)$ is a compact metric space, for example with the metric $\rho, 
\rho(V,W)=||P_V-P_W||$, where $P_V$ is the orthogonal projection onto $V$ and $||\cdot||$ is the operator norm.
 
\begin{df}
\label{vgr}For $L \in G(2n)$, denote
$$V_L=\{p \in \hn:p' \in L\}=L\times\T.$$
Every such $V_L$ is a homogeneous subgroup and it will be called vertical. 
\end{df}
\begin{df}
\label{vertra} For $d \in [1,2n]$ let
$$\V_{d}=\{V_L:L \in G(2n,d-1)\}.$$
Each vertical subgroup in $\V_{d}$ is a $d$-subgroup with metric (Hausdorff) 
dimension $d+1$. Moreover, let
$$\tr\V_{d}=\{a \cdot V_L:a \in \hn ,L \in G(2n,d-1)\}.$$
\end{df}

The homogeneous subgroups of $\hn$ which are not vertical are called horizontal. They are linear subpaces 
of $H$, that is, they belong to $G(2n)$. A subspace $L\in G(2n)$ is a (homogeneous) subgroup if and only if $A(p,q)=0$ for all $p,q\in L$. 
\begin{df}
\label{hortra}
For $d \in [0,2n]$ let
$$\mathcal{W}_d=\{G \subset \hn:G\text{ is a horizontal subgroup of } \hn \text{ such that }G \in G(2n,d)\},$$
and
$$\tr \mathcal{W}_d=\{a\cdot G:a \in \hn, G \in \mathcal{W}_d\}.$$
\end{df}

We denote by $Gr(n,m)$ the set of homogeneous subgroups of $\hn$ of Hausdorff dimension $m$. It is a closed 
subset of $G(2n+1)$. The Haar measures of $V\in Gr(n,m)$ are the positive constant multiples of 
$\mathcal{H}^m \lfloor V$, the restriction of the $m$-dimensional Hausdorff measure $\H^m$ to $V$. We denote the set of all such Haar measures by $\mathcal H(n,m)$.

\begin{lm}
\label{subgroup}
Let $L\in G(2n)$ such that $L$ is not a subgroup of $\hn$. Then there exists $p\in L$ such that
every sequence $(r_i)$ of positive numbers tending to $0$ has a subsequence $(r_{i_j})$ such that for 
some vertical subgroup $M$ with $\dim M=\dim L$,
$$\delta_{1/r_{i_j}}(p^{-1}L)\to M\ \text{as}\ j\to\infty.$$
\end{lm}

\begin{proof}
Since $L$ is not a subgroup there exist 
$p,q\in L$ such that  $p^{-1}q\notin L$. That $p,q \in L,p^{-1}q\notin L$ means that $A(p,q)\not=0.$
Let $p_i=p+r_i^2q\in L$. Then 
$$\delta_{1/r_i}(p^{-1}\cdot p_i)=(r_iq',-A(p,q))\to(0,-A(p,q))\in\T\setminus\{0\}.$$
We can take a subseqence $(r_{i_j})$ of $(r_i)$ such that the linear subspaces $\delta_{1/r_{i_j}}(p^{-1}L)$ converge
to some linear subspace $M$ of $\R^{2n+1}$. Then $\T\subset M\subset V_L$ and 
$\dim M=\dim \delta_{1/r_{i_j}}(p^{-1}L)=\dim L$ for all $j$. 
\end{proof}

\begin{df}
\label{rk}
The $s$-Riesz kernels in $\hn$, $s\in(0,2n+2]$, are defined as
$$R_s(p)=\left(R_{s,1}(p),\dots,R_{s,2n+1}(p)\right)$$
where
$$R_{s,i}(p)=\frac{p_i}{\|p\|^{s+1}} \text{ for }i=1,\dots,2n$$
and
$$R_{s,2n+1}(p)=\frac{p_{2n+1}}{\|p\|^{s+2}}.$$

\end{df}
Notice that these kernels are antisymmetric,
$$R_s(p^{-1})=(R_s(p))^{-1},$$
and $s$-homogeneous,
$$R_s(\delta_r(p))=\delta_{\frac{1}{r^s}}(R_s(p)).$$

\begin{df}
\label{rtr}
For a Radon measure $\mu$ in $\hn$, define the truncated $s$-Riesz transforms for $f \in L^1(\mu)$ by
$$\mathcal{R}^{\varepsilon}_s(f)(p)=\left( \rsei (f)(p) \right)_{i=1}^{2n+1},$$
where
$$ \rsei (f)(p)= \int_{\hn\setminus B(p,\varepsilon)}R_{s,i}(p^{-1}\cdot q)f(q)d \mu q.$$
The maximal $s$-Riesz transform is given by $$\mathcal{R}^{*}_s(f)(p)=\left(\rsim(f)(p)\right)_{i=1}^{2n+1}$$
where
$$\rsim (f)(p)=\sup_{\varepsilon >0}|\rsei (f) (p)| \text{ for }i=1,\dots,2n+1.$$
The maximal operator $\rsm$ is said to be bounded in $L^2(\mu)$
if the coordinate maximal operators  $\rsim$ are bounded
in $L^2(\mu)$ for all $i=1,\dots,2n+1$.
\end{df}

\begin{rem}
\label{T1homog} As an application of the T1 theorem in spaces of
homogeneous type, see \cite{dh}, it follows that if $m \in \N \cap
[1,2n+2]$ the maximal $m$-Riesz transforms $\mathcal{R}^{*}_m$ are
bounded in $L^2(\mu)$ for all $\mu\in\mathcal H(n,m)$.
\end{rem}

Let $\mu$ be a Radon measure in $\hn$. The image $f_\#\mu$  under a map $f:\hn \to \hn$ is the measure 
on $\hn$ defined by 
$$f_\#\mu(A)=\mu\big(f^{-1}(A)\big)\ \text{for all }\ A \subset \hn.$$
For $a\in \hn$ and $r>0$, $T_{a,r}: \hn \to \hn$ is defined for all $p\in \hn$ by
$$T_{a,r}(p)= \delta_{1/r}(a^{-1}\cdot p).$$

\begin{df}\label{tanm}
Let $\mu$ be a Radon measure on $\hn$. 
We say that $\nu$ is a \emph{tangent measure} of $\mu$ at $a\in \hn$ if $\nu$ is a Radon measure on 
$\hn$ with $\nu(\hn)>0$ and there are positive numbers $c_i$ and 
$r_i$, $i=1,2,\dots$, such that $r_i\to 0$ and
$$c_iT_{a,r_i\#}\mu \to \nu\ \text{weakly as}\ i\to\infty.$$ 
We denote by $\tanm(\mu,a)$ the set of all tangent measures of $\mu$ at a.
\end{df}

The numbers $c_i$ are normalization constants which are needed to keep $\nu$ non-trivial and locally finite. 
Often one can use $c_i=\mu(B(a,r_i))^{-1}$. The following lemma follows as in Remark 14.4 in \cite{M}.

\begin{lm}
\label{tandens}
Let $\mu$ be a Radon measure on $\hn$ and $s>0$ such that for $\mu$ a.e. $p \in \hn$,
\begin{equation*}
 0<\liminf_{r \ra 0} \frac{\mu(B(p,r))}{r^s} \leq
\limsup_{r \ra 0} \frac{\mu(B(p,r))}{r^s}<\infty.
\end{equation*}
Then for $\mu$ a.e. $ a\in\hn$ for every $\nu\in\tanm(\mu,a), 0\in\spt\nu$, and every sequence $(r_k)$ of 
positive numbers tending to $0$ has a subsequence $(r_{k_i})$ such that for some positive number $c$,
$$\nu=c\lim_{i \ra \infty} r_{k_i}^{-s}T_{a,r_{k_i}\sharp} \mu.$$ 
\end{lm}

\begin{df}\label{ittan}
Let $\mu$ be a Radon measure on $\hn$. 
We say that $\nu$ is an \emph{iterated tangent measure} of $\mu$ at $a\in \hn$ if there are Radon measures 
$\nu_1,\dots,\nu_m$ and points $p_i\in\spt\nu_i,i=1,\dots,m-1$, such that $\nu=\nu_m$ and
$$\nu_1\in \tanm(\mu,a),\nu_2\in \tanm(\nu_1,p_1),\dots, \nu_m\in \tanm(\nu_{m-1},p_{m-1}).$$
We denote by $\ittanm(\mu,a)$ the set of all iterated tangent measures of $\mu$ at a.
\end{df}

\begin{lm}
\label{ittanm}
Let $\mu$ be a Radon measure on $\hn$. Then for $\mu$ a.e. 
$a\in\hn$, $\ittanm(\mu,a)\subset\tanm(\mu,a)$.
\end{lm}
\begin{proof}
By a result of
Preiss in $\R^n$ (see \cite{P} or \cite{M}) generalized to metric groups in \cite{ms} the following is true: for 
$\mu$ a.e. $a\in\hn, \sigma\in\tanm(\mu,a)$ whenever $\nu\in\tanm(\mu,a), p\in\spt\nu$ and $\sigma\in\tanm(\nu,p)$. 
Let $a$ be such a point and 
$$\nu_1\in \tanm(\mu,a),\nu_2\in \tanm(\nu_1,p_1),\dots, \nu_m\in \tanm(\nu_{m-1},p_{m-1})$$
with $p_i\in\spt\nu_i$ for $i=1,\dots,m-1$. Choosing above $\nu=\nu_1$ and $\sigma=\nu_2$ we have $\nu_2\in\tanm(\mu,a)$. 
Further, choosing $\nu=\nu_2$ and $\sigma=\nu_3$ we have $\nu_3\in\tanm(\mu,a)$. Continuing this, we get 
$\nu_m\in\tanm(\mu,a)$.
\end{proof} 

\begin{rem}
We would like to thank Enrico Le Donne for the observation that the above lemma is valid. It was also proved and used in
a different setting in \cite{akl}.
\end{rem}

\section{Tangent measures and $s$-Riesz transforms}
\begin{thm}
\label{mthm} Let $s\in (0,2n+2)$ and let $\mu$ be a Radon measure in
$\hn$ satisfying for $\mu$ a.e. $p \in \hn$,
\begin{equation}
\label{dens} 0<\liminf_{r \ra 0} \frac{\mu(B(p,r))}{r^s} \leq
\limsup_{r \ra 0} \frac{\mu(B(p,r))}{r^s}<\infty .
\end{equation}
and 
\begin{equation}
\label{maxbound}
\sup_{0<\varepsilon <1}\left \| \left(\int_{B(p,1)\setminus B(p,\varepsilon)}
R_{s,i}(p^{-1}\cdot q)d\mu q \right)_{i=1}^{2n+1}\right\| < \infty.
\end{equation}
Then
\begin{enumerate}
\item $s $ is an integer in $[1,2n+1]$,
\item for $\mu$-a.e. $a \in \hn$, the set of tangent measures of $\mu$ at $a$, $\operatorname{Tan}(\mu,a)$,
contains measures in $\mathcal H(n,s)$.
\end{enumerate}
\end{thm}

Note that (\ref{maxbound}) is satisfied if $\rsm$ is bounded in $L^2(\mu)$.

\begin{proof} The assumption (\ref{maxbound}) is equivalent with,
\begin{equation}
\label{hmaxbound}
\sup_{0<\varepsilon <1} \left| \int _{B(p,1)\setminus B(p,\varepsilon)}
\frac{(p^{-1}\cdot q)_i}{\|p^{-1}\cdot q\|^{s+1}}d \mu q\right| < \infty \text{ for }i=1,..,2n
\end{equation}
and
\begin{equation}
\label{vmaxbound}\sup_{0<\varepsilon <1} \left| \int _{B(p,1)\setminus B(p,\varepsilon)}
\frac{(p^{-1}\cdot q)_{2n+1}}{\|p^{-1}\cdot q\|^{s+2}}d \mu q\right| < \infty.
\end{equation}
\begin{lm}
\label{1lm} Under the assumptions of Theorem \ref{mthm}, for $\mu$ a.e. $a \in \hn$  every $\nu \in \tanm(\mu, a)$ is 
an $s$-AD regular measure and there is $M<\infty$ such that,
\begin{equation}
\label{tangbound}
\sup_{0<r <R<\infty}\left \| \left(\int_{B(x,R)\setminus B(x,r)}
R_{s,i}(x^{-1}\cdot q)d\nu q \right)_{i=1}^{2n+1}\right\| \leq M\ \text{for all}\ x \in \spt \nu.
\end{equation}
\end{lm}
\begin{proof} One can prove that when (\ref{dens}) is satisfied for $\mu$ a.e $p \in \hn$, then for 
$\mu$ a.e. $a \in \hn$
every $\nu \in \tanm (\mu, a)$ is $s$-AD regular in an analogous way as in \cite{M}, p.190. Furthermore the
proof of this statement has very similar reasoning with the proof of (\ref{tangbound}).

To prove (\ref{tangbound}), let $\e >0$ and set
$$B=\{p \in \spt \mu : \sup_{0<\delta <1}
\left \| \left(\int_{B(p,1)\setminus B(p,\delta)}R_{s,i}(p^{-1}\cdot q)d\mu q \right)_{i=1}^{2n+1}\right\| < M\}$$
where $M$ is a positive constant that can be chosen such that $\mu(\hn \setminus B)< \e$. 
Furthermore it is enough to consider the
density points of $B$, i.e. the points $a \in B$ such that
$$\lim_{r \ra 0} \frac{\mu (B(a,r)\setminus B)}{r^s}=0.$$ 
For such a point $a \in B$ let $\nu \in \tanm (\mu,a)$. Then, by Lemma \ref{tandens},
there exist a positive number $c$ and a sequence of positive reals $(r_i)$ such that $r_i \ra 0$ and
$$\nu=\lim_{i \ra \infty} cr_i^{-s}T_{a,r_i,\sharp} \mu.$$

Let $x \in \spt \nu$. As in the proof of Lemma \ref{ittan} there exists a sequence $(a_i) \in B$ such that
\begin{equation}
\label{lm1eq} x_i:=\delta_{\frac{1}{r_i}}(a^{-1} \cdot a_i) \ra x.
\end{equation}

Now let $0<r<R<\infty$ be such that $\nu(\partial B(x,r))=\nu(\partial B(x,R))=0$, which is true for all but countably many 
$0<r<R<\infty$. For $j=1,..,2n$,
\begin{equation*}
\begin{split}
&\left|  \int _{B(x,R)\setminus B(x,r)} \frac{(x^{-1}\cdot y)_j}{\|x^{-1}\cdot y\|^{s+1}}d \nu y\right|\\
&=\lim_{i \ra \infty}\left| \int _{B(x_i,R)\setminus B(x_i,r)} \frac{(x_i^{-1}\cdot y)_j}{\| x_i^{-1}\cdot y\|^{s+1}}d \nu y\right|\\
&=\lim_{i \ra \infty}\left|\frac{1}{r_i^s} \int _{B(a \cdot \delta_{r_i}(x_i),Rr_i)\setminus B(a \cdot \delta_{r_i}(x_i),rr_i)} \frac{(x_i^{-1}\cdot \delta_{\frac{1}{r_i}}(a^{-1}\cdot y))_j}{\|(x_i^{-1}\cdot \delta_{\frac{1}{r_i}}(a^{-1}\cdot y)\|^{s+1}}d \mu y\right|\\
&=\lim_{i \ra \infty}\left| \int _{B(a_i,Rr_i)\setminus B(a_i,rr_i)}
\frac{(a_i^{-1}\cdot y)_j}{\| a_i^{-1}\cdot y\|^{s+1}}d \mu y\right|
\leq 2M,
\end{split}
\end{equation*}
where we used that $x_i^{-1} \cdot \delta_{\frac{1}{r_i}} (a^{-1} \cdot y) = 
\delta_{\frac{1}{r_i}} (a_i^{-1} \cdot y)$, $a_i=a \cdot \delta_{r_i}(x_i)$ and $a_i \in B$. In the second equality 
we have actually a double limit, but it is easily checked that it can be expressed as a single limit.

In a similar manner,
\begin{equation*}
\begin{split}
&\left|  \int _{B(x,R)\setminus B(x,r)} \frac{(x^{-1}\cdot y)_{2n+1}}{\|x^{-1}\cdot y\|^{s+2}}d \nu y\right|\\
&=\lim_{i \ra \infty}\left| \int _{B(a_i,Rr_i)\setminus B(a_i,rr_i)}
\frac{(a_i^{-1}\cdot y)_{2n+1}}{\| a_i^{-1}\cdot y\|^{s+2}}d \mu
y\right| \leq 2M^2,
\end{split}
\end{equation*}
By approximation, these estimates for $j =1,...,2n+1$ hold for all
$0<r<R<\infty$.
\end{proof}
\begin{pr}
\label{vertprop} Let $L\in G(2n,d)$ and let $\nu$ be an $s$-AD-regular measure in
$\hn$ such that (\ref{tangbound}) holds.
\begin{enumerate}
\item If $\spt \nu\subset V_L$ and $\spt\nu \not= V_L$, 
\end{enumerate}
then there exist $b\in\spt\nu$ and  
$\sigma \in \tanm (\nu,b)$ such that either 
$\spt \sigma \subset V_M$ for some $M \in G(2n,d-1)$ with $M\subset L$ or $\spt \sigma\subset L$.
\begin{enumerate}
\item [(ii)] If $\spt\nu\subset L$ and $\spt \nu\not=L$, 
\end{enumerate}
then there exist $b\in\spt\nu$ and $\sigma \in \tanm (\nu,b)$ 
such that $\spt \sigma \subset V_M$ for some  $M\in G(2n,d-1)$ with $M\subset L$.
\end{pr}
\begin{proof} Given $b\in\spt\nu$ and $\pi \in \tanm (\nu,b)$, consider the following two statements:
\begin{enumerate}
\item [(iii)]For some $M \in G(2n,d-1)$ and every $\delta \in (0,1)$ there exists $\e>0$ such that
$\spt \pi \cap B(\e) \cap \{p \in \hn : d_E(p',M)>\delta \| p\|\}=\emptyset$,
\item [(iv)] For every $\delta \in (0,1)$ there exists $\e>0$ such that
$\spt \pi \cap B(\e) \cap \{p \in \hn : \sqrt {|p_{2n+1}|}>\delta \| p\|\}=\emptyset.$
\end{enumerate}
We shall verify that in order to prove the proposition, it is enough to prove that 
there exist $b\in\spt\nu$ and $\pi \in \tanm (\nu,b)$ such that 
(iii) or (iv) holds in the case (i) and (iii) holds in the case (ii).
In order to see that this, suppose that $b\in\spt\nu$ and $\pi \in \tanm (\nu,b)$. Then, recalling Lemma \ref{tandens},
$0\in\spt\pi$.  
Let $\sigma \in \tanm (\pi,0)$ be such that 
$\sigma= \lim_{i \ra\infty}\frac{1}{r_i^s}T_{0,r_i,\sharp} \pi$ for some sequence
of positive reals $(r_i)$ such that $r_i \ra 0$. Consider first the case (i) and suppose that (iii) holds. Then for all 
$R>0,\delta \in (0,1)$ and $G_{R,\delta}=\{p \in U(R):d_E(p',M)>\delta \|p\|\}$,
\begin{equation*}
\begin{split}
\sigma(G_{R,\delta})&\leq \liminf_{i \ra \infty} \frac{1}{r_i^s} \pi (T_{0,r_i}^{-1}(G_{R,\delta}))\\
&\leq \liminf_{i \ra \infty}\frac{1}{r_i^s} \pi (\{p \in \hn: \delta_{\frac{1}{r_i}}(p) \in B(R) 
\text{ and } d_E \left(\frac{p'}{r_i},M\right)>\delta \|\delta_{\frac{1}{r_i}}(p)\|\})\\
&=\liminf_{i \ra \infty} \frac{1}{r_i^s}\pi( B(r_iR) \cap \{p \in \hn : d_E(p',M)> \delta \|p\|\})\\
&=0,
\end{split}
\end{equation*}
where the last equality follows by (iii). Since $\sigma (G_{R,\delta})=0$ for all $R>0$ and $\delta \in (0,1)$ 
we deduce that $\spt \sigma \subset V_M$. In the same way (iv) implies that 
$\spt \sigma \subset H$. Clearly $\spt\sigma\subset V_L$, and so $\spt \sigma \subset H\cap L$. 

As $\sigma$ is a tangent measure of $\pi \in \tanm (\nu,b)$ at 0, it is easy to check that 
$\sigma \in \tanm (\nu,b)$, whence we have verified in the case (i) the sufficiency of the conditions (iii) and (iv). In the same way (iii) suffices in the case (ii).

Suppose that $\spt \nu\subset V_L$ and $\spt\nu \not= V_L$, and 
by way of contradiction assume that for all $b \in \spt \nu$, all $\pi \in \tanm (\nu,b)$ and every $M \in
G(2n,d-1)$ there exist $\delta_{\pi,M},\delta_{\pi_H} \in
(0,1)$ such that for all $\e>0$
\begin{equation}
\label{sigvertnon}
B(\e)\cap \spt \pi \cap \{p \in \hn:d_E(p',M)>\delta_{\pi,M}\|p\|\}\neq \emptyset
\end{equation}
and
\begin{equation}
\label{sighornon}
B(\e)\cap \spt \pi \cap \{p \in \hn:\sqrt{|p_{2n+1}|}>\delta_{\pi_H}\|p\|\}\neq \emptyset.
\end{equation}

We proceed with a geometric lemma, but first we introduce some notation. We denote by $(\cdot,\cdot)$ 
the usual inner product in $\R^{2n}$. For $c\in\hn$, define
\begin{equation*}
\begin{split}
&a_j(c)=|c'|^2c_j+c_{2n+1}c_{n+j} \text{ for }j=1,\dots,n, \\
&a_j(c)=|c'|^2c_j-c_{2n+1}c_{j-n}\text{ for }j=n+1,\dots,2n,\\
&a(c)=(a_1(c),\dots,a_{2n}(c)),\\
&L(c)=\{y \in \R^{2n}:(a(c),y)=0\},\\
&V_L^+(c)=\{y \in \hn:(a(c),y)\geq0\},\\
&V_L^-(c)=\{y \in \hn:(a(c),y)\leq0\},\\
&H^+(c)=\{y\in\hn:y_{2n+1}c_{2n+1}\geq0\},\\
&H^-(c)=\{y\in\hn:y_{2n+1}c_{2n+1}\leq0\}.
\end{split}
\end{equation*}

\begin{lm}
\label{dimension} Let $L\in G(2n,d)$ and $c\in V_L\setminus\T$. Then  
$\dim L\cap L(c)<d$.
\end{lm}
\begin{proof}From the definition of $a(c)$ we see that $(a(c),c')=|c'|^4>0$. Hence $c'\in L\setminus L(c)$, 
which implies the lemma.
\end{proof}

\begin{lm}
\label{boundplanes} Let $L\in G(2n,d), b,c\in V_L,r>0$  and let the sequences 
$(p_i) \in V_L, (r_i) \in \R^+$ be such that
\begin{enumerate}
\item $b \in \partial B(c,r)$,
\item $d(p_i,c)\geq r$ for all $i \in \N$,
\item $\lim_{i \ra \infty} p_i=b$,
\item $\lim_{i \ra \infty} r_i=0$,
\item $\lim_{i \ra \infty}\delta_{\frac{1}{r_i}}(b^{-1}\cdot p_i)=p_*$.
\end{enumerate}
If $b^{-1}\cdot c \notin \T$ then $p_* \in V_L\cap V_{L(b^{-1}\cdot c)}^-$. If 
$b^{-1}\cdot c\in\T$, then $p_* \in H^-(b^{-1}\cdot c)$.
\end{lm}
\begin{proof} Replacing $c$ by $b^{-1}\cdot c$ and $p_i$ by $b^{-1}\cdot p_i$, we may assume that $b=0$.
First assume that $c \notin \T$. Then as $d(p_i,c)\geq r$ and $d(0,c)=r$
\begin{equation}
\label{1vertequ}
|p'_i-c'|^4+\left| p_{i,2n+1}-c_{2n+1}-A(p_i,c)\right|^2\geq r^4
\end{equation}
and
\begin{equation}
\label{2vertequ}
|c'|^4+|c_{2n+1}|^2=r^4.
\end{equation}
Here (\ref{1vertequ}) can be written as,
\begin{equation*}
\begin{split}
|p'_i|^4&+|c'|^4+2|p'_i|^2|c'|^2-4(p'_i,c')|p'_i|^2-4(p'_i,c')|c'|^2+4(p'_i,c')^2\\
&+|p_{i,2n+1}|^2+|c_{2n+1}|^2-2p_{i,2n+1}c_{2n+1}-2p_{i,2n+1}A(p_i,c)\\
&+2c_{2n+1}A(p_i,c)+A(p_i,c)^2\geq r^4.
\end{split}
\end{equation*}
After using (\ref{2vertequ}), dividing by $r_i$ and letting $i \ra \infty$ we obtain,
\begin{equation}
-2(a(c),p^*)=-2|c'|^2(p'_*,c')+c_{2n+1}A(p'_*,c')\geq0.
\end{equation}
Therefore $p_* \in V_{L(c)}^-$. So the lemma is proven in the case $c\notin\T$.

Now let $c \in \T$. As in the previous case, combining that $c'=0$, $d(p_i,c)\geq r$ and $d(0,c)=r$, 
we get
$$|p'_i|^4+p_{i,2n+1}^2-2p_{i,2n+1}c_{2n+1}\geq0.$$
After dividing by $r_i^2$ and letting $i \ra \infty$ we conclude that
$$c_{2n+1}p_{*,2n+1}\leq0.$$
As $c_{2n+1} \neq 0$, $p_* \in H^-(c)$.
\end{proof}

As $\spt \nu\subset V_L$ and $\spt\nu \not= V_L$ there
exist $b,c \in \hn$ and $r>0$ such that
\begin{equation*}
\begin{split}
b \in B(c,r) \cap \spt \nu,  \\
U(c,r) \cap \spt \nu =\emptyset.
\end{split}
\end{equation*}
If we have (ii) of Proposition \ref{vertprop}, that is, $\spt\nu\subset L$ and $\spt \nu\not=L$, 
we can take $c\in L$, and so $b^{-1}\cdot c\notin \T$.

We first consider the case when $b^{-1}\cdot c \notin \T$. We shall prove that if $\lambda \in \tanm (\nu,b)$, then 
\begin{equation}
\label{verthalf} \spt \lambda \subset V^-_{L(b^{-1}\cdot c)} \cap V_L\ \text{and}\ \dim L(b^{-1}\cdot c) \cap L<d.
\end{equation}
As $\lambda \in \tanm (\nu,b)$, by Lemma \ref{tandens}, there exist positive
numbers $C$ and $r_i, r_i \ra 0$, such that 
$\lambda =\lim_{i \ra\infty} C\frac{1}{r_i^s} T_{b,r_i\sharp} \nu$. Then for all $R>0,
\delta \in (0,1)$ and
$$G_{R,\delta}=U(R)\setminus V^-_{L(b^{-1}\cdot c)} \cap \{p \in \hn: d_E(p',L(b^{-1}\cdot c))>\delta \|p\|\},$$
we get
\begin{equation*}
\begin{split}
&C^{-1}\lambda(G_{R,\delta})\\
&\leq \liminf_{i \ra \infty}\frac{1}{r_i^s}\nu
(T_{b,r_i}^{-1}(G_{R,\delta})) \\
&=\liminf_{i \ra \infty} \frac{1}{r_i^s}\nu (U(b,r_i R) \cap \{p \in \hn:b^{-1}\cdot p\notin V^-_{L(b^{-1}\cdot c)}, 
 d_E((b^{-1}\cdot p)',L(b^{-1}\cdot  c))>\delta d(b,p)\}).
\end{split}
\end{equation*}
Therefore, in order to prove the inclusion in (\ref{verthalf}), it is enough to show that there exists 
some $\e >0$ such
that,
\begin{equation*}
\spt \nu \cap B(b,\e)\cap\{p \in \hn:b^{-1}p\in\ V^+_{L(b^{-1}\cdot c)}, d_E((b^{-1}\cdot p)',L(b^{-1}\cdot c))>\delta d(b,p)\} = \emptyset.
\end{equation*}
By way of contradiction suppose that there exists a sequence
$(p_i)\in \hn, p_i \ra b$, satisfying for all $i \in \N$,
\begin{enumerate}
\item $p_i \in \spt \nu\subset\hn\setminus U(c,r),$
\item $b^{-1}\cdot p_i \in V^+_{L(b^{-1}\cdot c)},$
\item $d_E((b^{-1}\cdot p_i)',L(b^{-1}\cdot c))>\delta d(b,p_i).$
\end{enumerate}
The new sequence $\delta_{\frac{1}{d(b^{-1},p_i)}}(b^{-1}\cdot p_i) \in B(1)$ has a
converging subsequence and for simplifying notation we write,
$$\delta_{\frac{1}{d(b^{-1},p_i)}}(b^{-1}\cdot p_i) \ra p_*.$$

Notice that by (ii) and (iii) $p_* \notin V^-_{L(b^{-1}\cdot c)}$. But by Lemma 
\ref{boundplanes}, $p_* \in V^-_{L(b^{-1} \cdot c)}$ and we have reached a contradiction. Recalling Lemma \ref{dimension}, (\ref{verthalf}) follows.

Combining (\ref{sigvertnon}) and (\ref{verthalf}) we obtain that
there exists $\delta=\delta_{\lambda,L(b^{-1}\cdot c)}\in (0,1)$ such that for all $r>0$
$$\spt \lambda \cap B(r) \cap V^-_{L(b^{-1}\cdot c)} \cap \{ p \in \hn : d_E(p',L(b^{-1}\cdot c))> \delta \|p\|\} \neq \emptyset.$$
Without loss of generality we can assume that $V^-_{L(b^{-1}\cdot c)}= \{p \in
\hn: p_{2n}>0\}$. Hence there exists a sequence $(x_i),x_i \ra 0,$
such that for all $i \in \N$,
$$x_i \in \spt \lambda \cap \{ p \in \hn : p_{2n}> \delta \|p\|\}.$$
Furthermore this sequence can be chosen to satisfy,
$$\|x_1\|>\|x_2\|>....$$
and the balls $B_i=B(x_i, \frac{\delta \|x_i\|}{2})$ can be assumed
to be disjoint and contained in $B(0,1)$. Then for all $k \in \N$, by the AD-regularity of $\lambda$,
\begin{equation*}
\begin{split}
\int_{B(0,1) \setminus
B(0,\frac{\|x_k\|}{2})}\frac{y_{2n}}{\|y\|^{s+1}}d \lambda y &\geq
\sum_{i=1}^{k} \int_{B_i}\frac{y_{2n}}{\|y\|^{s+1}}d \lambda y \\
&\geq \sum_{i=1}^{k} \frac{\frac{\delta
\|x_i\|}{2}}{\frac{3^{s+1}\|x_i\|^{s+1}}{2^{s+1}}} \lambda (B_i) \\
&\geq C\sum_{i=1}^{k} \frac{\|x_i\|\|x_i\|^s}{\|x_i\|^{s+1}}\\
&=Ck,
\end{split}
\end{equation*}
where $C$ is independent of $k$. Hence 
$$\lim_{k \ra \infty} \int_{B(0,1) \setminus
B(0,\frac{\|x_k\|}{2})}\frac{y_{2n}}{\|y\|^{s+1}}d \lambda y =
\infty.$$
This is a contradiction by Lemma \ref{1lm} as $0 \in \spt \lambda$.

We are now left to consider the case when $b^{-1}\cdot c \in \T$.  In an
identical way as in the proof of (\ref{verthalf}) we can show that
if $ \lambda \in \tanm (\nu, b)$,
\begin{equation}
\label{horhalf} \spt \lambda \subset H^-(b^{-1}\cdot c).
\end{equation}
Combining (\ref{sighornon}) and (\ref{horhalf}) we deduce that there
exists some $\delta=\delta_\lambda$ such that for all $r>0$,
\begin{equation}
\label{horfincontr} \spt \lambda \cap B(r) \cap \{ p \in \hn : p_{2n+1}>0, 
\sqrt{p_{2n+1}}> \delta \|p\|\} \neq \emptyset.
\end{equation}
assuming that $H^-(b^{-1}\cdot c)=\{p \in \hn:p_{2n+1}>0\}$. Finally in order to
complete the estimates, which are otherwise identical with the ones in the case where $b^{-1}\cdot c
\notin \T$,  we need the following simple lemma.

\begin{lm}
\label{horest} Let $x \in \hn$ and $\delta \in (0,1)$ such that $x_{2n+1}>0$ and 
$\sqrt{x_{2n+1}}>\delta \|x\|.$ Then for all $y \in
B(x,\frac{\delta^2 \|x\|}{100n})$, $$y_{2n+1} \geq \frac{\delta^2
\|x\|}{2}.$$
\end{lm}
\begin{proof} As,
$$|y_{2n+1}-x_{2n+1}| \leq \|y \cdot x^{-1}\|^2+2\left|\sum_{i=1}^n(y_i x_{i+n}- y_{i+n}x_i)\right|,$$
for $y \in B(x,\frac{\delta^2 \|x\|}{100n})$ we get,
\begin{equation*}
\begin{split}
|y_{2n+1}-x_{2n+1}| &\leq \|y \cdot
x^{-1}\|^2+2\left|\sum_{i=1}^n(y_i x_{i+n}-x_ix_{i+n}+x_ix_{i+n}
-y_{i+n}x_i)\right| \\
&\leq \|y \cdot x^{-1}\|^2+2\sum_{i=1}^n|y_i-x_i||x_{i+n}|+|x_{i+n}
-y_{i+n}||x_i|\\
&\leq \|y \cdot x^{-1}\|^2+2\sum_{i=1}^n 2 \|x\|\frac{\delta^2
\|x\|}{100n} \\
&\leq \frac{\delta^4 \|x\|^2}{(100n)^2}+\frac{\delta^2 \|x\|^2}{25}\\
&\leq \frac{\delta^2 \|x\|^2}{10}.
\end{split}
\end{equation*}
Therefore,
$$y_{2n+1}\geq x_{2n+1}-\frac{\delta^2 \|x\|^2}{10}>\frac{\delta^2 \|x\|^2}{2}.$$
\end{proof}
By (\ref{horfincontr}), there exists some sequence $(x_i),x_i \in
\spt \lambda$ such that $\sqrt{x_{i,2n+1}}>\delta \|x_i\|$ for all
$i \in \N$ and $x_i \ra 0$. As before we can assume that
$$\|x_1\|>\|x_2\|>....$$
and the balls $B_i=B(x_i, \frac{\delta^2 \|x_i\|}{100n})$ are
disjoint and contained in $B(0,1)$. Then for all $k \in \N$,
\begin{equation*}
\begin{split}
\int_{B(0,1) \setminus
B(0,\frac{\|x_k\|}{2})}\frac{y_{2n+1}}{\|y\|^{s+2}}d \lambda y &\geq
\sum_{i=1}^{k} \int_{B_i}\frac{y_{2n+1}}{\|y\|^{s+2}}d \lambda y \\
&\geq \sum_{i=1}^{k} \frac{\delta^2 \|x_i\|^2}{2\cdot 2^{s+2}\|x_i\|^{s+2}} \lambda (B_i) \\
&\geq C\sum_{i=1}^{k} \frac{\|x_i\|^2 \|x_i\|^s}{\|x_i\|^{s+2}}\\
&=Ck,
\end{split}
\end{equation*}
where we used Lemma \ref{horest} in the third inequality and $C>0$ is independent of $k$. As in the
previous case we have reached a contradiction by (\ref{1lm}). The proof of 
Proposition \ref{vertprop} is now complete.

\end{proof}

We now continue with the proof of Theorem
\ref{mthm}. We first apply for $\mu$ a.e. $a\in\hn$ Lemma \ref{1lm} to get an $s$-AD-regular $\nu\in\tanm(\mu,a)$ 
with the property (\ref{tangbound}). Then we apply to $\nu$, to its tangent measures, and so on, Proposition 
\ref{vertprop} sufficiently many times to find for $\mu$ a.e. $a\in\hn$ an integer 
$d\in[1,2n], L\in G(2n,d)$ and $\pi\in \ittanm(\mu,a)$ such that denoting $V=\spt\pi$, 
$V=V_L$ or $V=L$. If $V=L$ and $L$ is not a subgroup, we take a tangent measure $\rho\in\tanm(\pi,p)$ at a point $p\in L$ as in Lemma \ref{subgroup}. Then by Lemma \ref{subgroup} its support is $V_M$ for some $M\in G(2n,d-1)$. Thus in 
any case we find for $\mu$ a.e. $a\in\hn$ an integer 
$d'\in[1,2n], L\in G(2n,d')$ and $\rho\in \ittanm(\mu,a)$ such that denoting $V=\spt\rho$, 
$V=V_L$ or $V=L$ and $L$ is a homogeneous subgroup.
As $\rho$ is AD-regular, $s$ must be the Hausdorff dimension of $V$, that is, $s=d$ or $s=d+1$ 
and so $s$ is an integer. Furthermore, by the AD-regularity, 
$\rho$ is absolutely continuous with respect to $\mathcal{H}^s \lfloor V$ with bounded density $h$. 
Taking a tangent measure at a point of approximate continuity of $h$ we obtain 
$\sigma\in \tanm (\rho,b)\subset \ittanm(\mu,a)$ such that $\sigma=\mathcal{H}^s \lfloor V\in\mathcal H(n,s)$. 
Applying still Lemma \ref{ittanm} we find that $\mu$ has almost everywhere tangent measures in $\mathcal H(n,s)$.
This completes the proof of Theorem \ref{mthm}.

\end{proof}

\section{Riesz transforms and self-similar sets in $\hn$}

In this section we consider certain families of self-similar sets in
$\hn$ and we discuss their relations with the Riesz transforms that
we introduced earlier.

\begin{df}
\label{similitudes} Let $Q=[0,1]^{2n} \subset \R^{2n}$ and $r \in
(0,\frac{1}{2}).$ Let $z_j \in \R^{2n}, j=1,...,2^{2n},$ be distinct
points such that $z_{j,i} \in \{0,1-r\}$ for all $j=1,..,2^{2n}$ and
$i=1,..,2n$. We consider the following $2^{2n+2}$ similarities
depending on the parameter $r$,
\begin{equation*}
\begin{split}
S_j&=\tau_{(z_j,0)} \delta_r, \text{ for }j=1,...,2^{2n}, \\
S_j&=\t_{(z_{\lfloor j \rfloor_ {2^{2n}}},\frac{1}{4})}\delta_r, \text{ for }j=2^{2n}+1,...,2\cdot2^{2n}, \\
S_j&=\t_{(z_{\lfloor j \rfloor_ {2 \cdot 2^{2n}}},\frac{1}{2})}\delta_r, \text{ for }j=2 \cdot 2^{2n}+1,...,3\cdot2^{2n}, \\
S_j&=\t_{(z_{\lfloor j \rfloor_ {3 \cdot 2^{2n}}},\frac{3}{4})}\delta_r, \text{ for }j=3 \cdot 2^{2n}+1,...,2^{2n+2}, \\
\end{split}
\end{equation*}
where $\lfloor j \rfloor_m:=j \mod m$.
\end{df}

\begin{thm}
\label{selfsimilar} Let $r \in (0,\frac{1}{2})$ and
$\mathcal{S}_r=\{S_1,...,S_{2^{2n+2}}\}$ where the $S_j$'s are the
similarities of Definition \ref{similitudes}. Let $K_r$ be the
self-similar set defined by,
$$K_r=\bigcup_{j=1}^{2^{2n+2}}S_j (K_r).$$
Then the the sets $S_j(K_r)$ are disjoint for $j=1,..,2^{2n+2}$,
$$K_r = \bigcap_{k=1}^{\infty} \bigcup _{j_1,..,j_k =1}^{2^{2n+2}} S_{j_1}\circ..\circ
S_{j_k}(K_r)$$ and
$$0< \mathcal{H}^{a}(K_r)< \infty \text{ with }a=\frac{(2n+2)\log(2) }{\log (\frac{1}{r})}.$$
\end{thm}
\begin{proof} It is enough to find some set $R \supset K$ such that for all
$j=1,..,2^{2n+2}$,
\begin{enumerate}
\item $S_j(R) \subset R$ and
\item the sets $S_j(R)$ are disjoint,
\end{enumerate}
see \cite{s} and \cite{br}. Using an idea of Strichartz from \cite{St} we show that there exists a
continuous function $\f : Q \ra \R$ such that the set
$$R=\{q \in \hn : q' \in Q \text{ and } \f (q') \leq q_{2n+1} \leq \f (q')+1\}$$
satisfies (i) and (ii).

This will follow immediately if we find some continuous $\f : Q \ra
\R$ which satisfies for all $j=1,..,2^{2n}$,
\begin{equation}
\label{ss1} \tau_{(z_j,0)} \delta_r (R) = \{q \in \hn : q' \in Q_j
\text{ and } \f (q') \leq q_{2n+1} \leq \f (q')+r^2\},
\end{equation}
where $Q_j=\t_{(z_j,0)}(\delta_r(Q))$. Since
\begin{equation*}
\begin{split}
\tau_{(z_j,0)} \delta_r (R) = \{p \in \hn : p' \in Q_j \text{ and } 
&r^2\f (\frac{p'-z_j}{r})-2\sum_{i=1}^{n}
(z_{j,i}p_{i+n}-z_{j,i+n}p_i) \leq p_{2n+1}\\
&\leq r^2\f (\frac{p'-z_j}{r})-2\sum_{i=1}^{n}
(z_{j,i}p_{i+n}-z_{j,i+n}p_i)+r^2 \}.
\end{split}
\end{equation*}
proving (\ref{ss1}) amounts to showing that
\begin{equation}
\label{ss2} \f (w)= r^2\f (\frac{w-z_j}{r})-2\sum_{i=1}^{n}
(z_{j,i}w_{i+n}-z_{j,i+n}w_i) \text{ for }w \in Q_j,j=1,..,2^{2n}.
\end{equation}

As usual for any metric space $X$, denote $C(X)=\{f:X\ra \R\text{ and }f \ \text{is continuous}\}$. Let $B=\cup_{j=1}^{2^{2n}} Q_j$ and $L: C(B) \ra C(Q)$ be a linear
extension operator such that
\begin{equation*}
L (f) (x)= f(x) \text { for } x \in B
\end{equation*}
and
$$\|L(f)\|_{\infty}= \|f\|_{\infty}.$$

Since the $Q_j$'s are disjoint the operator $L$ can be defined
simply by taking $\e>0$ small enough and letting

\begin{equation*}
L(f)(x)=
\begin{cases} f(x) & \text{ when }x \in B,
\\
\dfrac{\e - \dist(x,B)}{\e} f (\tilde{x})&\text{ when
}0<\dist(x,B)<\e,
\\
0 & \text{ when } \dist(x,B) \geq \e,
\end{cases}
\end{equation*}
where $\tilde{x} \in B$ and $\dist(x,B)=d(x,\tilde{x}).$

Furthermore define the functions $h: B \ra \R $, $\tilde{f} :B \ra \R$,
$$h(w)=-2\sum_{i=1}^{n}
(z_{j,i}w_{i+n}-z_{j,i+n}w_i) \text{ for }w \in Q_j,$$
$$\tilde{f} (w)= r^2 f (\frac{w-z_j}{r}) \text { for }f \in C(Q),w \in Q_j ,$$
and the operator $T:C(B) \ra C(Q)$ as,
$$T(f)=L(\tilde{f} + h).$$
Then
$$T(f)(w)=r^2 f (\frac{w-z_j}{r})-2\sum_{i=1}^{n}
(z_{j,i}w_{i+n}-z_{j,i+n}w_i) \text{ for }w \in Q_j,$$ and for $f,g
\in C(B)$
\begin{equation*}
\|Tf-Tg\|_\infty
=\|L(\tilde{f}-\tilde{g})\|_\infty=\|\tilde{f}-\tilde{g}\|_\infty
\leq r^2 \|f-g\|_\infty.
\end{equation*}
Hence $T$ is a contraction and it has a unique fixed point $\f$
which satisfies (\ref{ss2}).
\end{proof}

Let $\mathcal{S}=\{S_1,..,S_N\}$  be an iterated function system (IFS) of similarities of the
form
\begin{equation}
\label{heissim}S_i=\t_{q_i}\circ \delta_{r_i}
\end{equation}
for $q_i \in \hn,i=1,..,N$ and $0<r_1 \leq...\leq r_N <1$. The IFS $\mathcal{S}$ is said to satisfy the open set condition if there exists a bounded non-empty open set $O\subset \hn$ such that $\cup_{i=1}^{N} S_i (O)\subset O$ and $S_i (O) \cap S_j (O)=\emptyset$ for $i \neq j$. It
follows by \cite{br} that if $K \subset \hn$ is the invariant set
with respect to $\mathcal{S}$, and $\mathcal{S}$ satisfies the open
set condition then,
$$0<\mathcal{H}^a(K)< \infty$$ where $a$ is given by $$\sum_{i=1}^N r^a_i=1.$$

Recalling Definitions \ref{vertra} and \ref{hortra} our next result reads as follows.
\begin{thm}
\label{rig} Let $\mathcal{S}$ be an IFS of contractive similarities
as in (\ref{heissim}) that satisfies the open set condition, and
$$0<\mathcal{H}^m(K)<\infty \ \text{where} \ m \in \N \cap [1,2n+1].$$
If $K$ is not contained in any $F
\in \tr \V_{m-1} \cup \tr \mathcal{W}_m$ then for all $k \in K$, every $\nu \in \tanm (\H^m
\lfloor K,k)$ and every $G \in \V_{m-1} \cup \mathcal{W}_m$ satisfy
$$\spt \nu \setminus G \neq \emptyset. $$ Therefore $\tanm (\H^m \lfloor K,k) \cap \mathcal{H} (n,m)=\emptyset$.
\end{thm}

\begin{proof} The proof follows the reasoning developed in \cite{M1},
where a rigidity result for Euclidean self-similar sets is proven.
Assume without loss of generality that $\diam(K)=1$, let $k \in K$
and $G \in \V_{m-1} \cup \W_m$. As $K$ is not contained in any $F \in \tr
\V_{m-1} \cup \tr \W_m$ there exist $E \subset K$,
$$E=\{a_1,..,a_{m+2}\},$$
and $0<\r <1$ such that for all $F \in \tr \V_{m-1} \cup \tr \W_{m}$ there exist
$a_i=a_i (F),i=1,..,m+2$, such that,
\begin{equation}
\label{distpla} \dist (a_i,F)> \r.
\end{equation}

Furthermore there exist $0<r_0<\frac{\r}{2}$ and $\eta >0$, such
that
\begin{equation}
\label{lodens} \H^d (K \cap B(a_i,r_0)) \geq \eta  \text{ for }
i=1,\dots,m+2.
\end{equation}
Recall that the contraction ratios satisfy,
$$0<r_1 \leq...\leq r_N <1$$ and let $0<r<r_1$. For any word $\a=(a_1,\dots,a_m),a_i \in \{1,\dots,N\},m\in \N,$ denote $r_\a=r_{a_1}\dots r_{a_m}$ and $A_\a=S_\a (A)$ for any set $A \subset \hn$. Let $\a$ be a
minimal word such that
$$k \in K_\a\subset B(k,\frac{r}{2}).$$
Then it follows that
$$r_\a \leq \frac{r}{2}\leq \frac{r_\a}{r_1}.$$

Notice also that for any map of the form,
$$S=\t_q \circ \delta _r \text{ for } q\in \hn,r>0,$$
$$S(F), S^{-1}(F) \in \tr \V_m \text{ if }F \in \tr \V_m, m \in [1,2n]$$
and
$$S(F), S^{-1}(F) \in \tr \W_m \text{ if }F \in \tr \W_m, m \in [0,2n].$$
Therefore $S^{-1}_\a(k \cdot G) \in \tr \V_{m-1}\cup \tr \W_{m}$ and by
(\ref{distpla}) there exist $a_i \in E,i=1,..,m+2,$ such that
$$\dist (a_i,S^{-1}_\a(k \cdot G))> \r.$$
Therefore,
\begin{equation}
\label{conest} B(S_\a a_i, r_\a r_0)\subset B(k,r) \setminus \{x \in
\hn : \dist(x, k \cdot G)\leq \frac{\r r_\a}{2}\}.
\end{equation}
This follows because for $x \in B(S_\a a_i, r_\a r_0),$
$$d(x,k) \leq d(x, S_\a a_i)+d(S_\a a_i,k)\leq r_\a r_0+r_\a \leq r,$$
and
\begin{equation*}
\begin{split}
\dist(x, k \cdot G) &\geq \dist(S_\a a_i, k \cdot G)-d(x, S_\a a_i)
\\
&=r_a \dist(a_i,S^{-1}_\a  (k \cdot G))-d(x, S_\a a_i)>\frac{\r
r_\a}{2}.
\end{split}
\end{equation*}
Hence for $\delta =\dfrac{\r r_1}{4},$ by (\ref{conest}) and
(\ref{lodens}),
\begin{equation*}
\begin{split}
\H^m (K \cap B(k,r) \setminus X(k,G,\delta))&\geq \H^m (K \cap
B(k,r)
\setminus \{x \in \hn:\dist (x, k\cdot G) \leq \delta r \}) \\
& \geq \H^m (K_\a \cap B(S_\a a_i,r_\a r_0)) \\
& = r^m_\a \H^m (K \cap B(a_i,r_0)) \geq \eta r^m_\a  \\
& \geq \eta (\frac{r_1}{2})^m r^m.
\end{split}
\end{equation*}

Therefore there exist $C>0$ and $\delta \in (0,1)$ such that if
$0<r<r_1$
\begin{equation}
\label{tanconest} \frac{\H^m (K \cap B(k,r) \setminus
X(k,G,\delta))}{r^m}>C.
\end{equation}

Finally let $\nu \in \tanm (\H^m \lfloor K, k)$ and recalling Lemma \ref{tandens},
$$\nu = c\lim_{i \ra \infty}\frac{1}{r_i^m} T_{k,r_i,\sharp }\H^m \lfloor K,$$
for some positive numbers $c,(r_i)$ with $r_i \ra 0$. Then
\begin{equation*}
\begin{split}
\nu (B(0,1) \setminus X(0,V,\delta)) &\geq c\limsup_{i \ra
\infty}\frac{1}{r_i^m} \H^m  (K \cap T_{k,r_i}^{-1}(B(0,1)\setminus
X(0,G,\delta))) \\
&=c\limsup_{i \ra \infty}\frac{1}{r_i^m} \H^m  (K \cap
B(k,r_i)\setminus X(k,G,\delta))) \\
&\geq cC.
\end{split}
\end{equation*}
Therefore $\spt \nu \not\subset G$ and the proof is complete.
\end{proof}

As an immediate corollary of Theorems \ref{mthm} and \ref{rig} we
obtain.
\begin{cor}
Let $\mathcal{S}$ be an IFS of contractive similarities as in
(\ref{heissim}) that satisfies the open set condition, and
$$0<\mathcal{H}^m(K)<\infty \text{ where }m \in \N \cap [1,2n+1].$$
If $K$ is not contained in any translated $(m-1)$-vertical or $m$-horizontal group, i.e. $K \not\subset G$ for all $G \in \tr \V_{m-1} \cup \tr\W_m$, then the $m$-Riesz transforms $\mathcal{R}^{*}_m$
are not bounded in $L^2(\H^m \lfloor K)$.

In particular if $\mu$ is the natural measure on the $m$-dimensional
Cantor-like sets of Theorem \ref{selfsimilar}, then the $m$-Riesz
transforms $\mathcal{R}^{*}_m$ are not bounded in $L^2(\mu)$.
\end{cor}

\vspace{1cm}
\begin{footnotesize}
{\sc Department of Mathematics and Statistics,
P.O. Box 68,  FI-00014 University of Helsinki, Finland,}\\
\emph{E-mail addresses:} \verb"vasileios.chousionis@helsinki.fi",
\verb"pertti.mattila@helsinki.fi"

\end{footnotesize}

\end{document}